\documentclass[12pt]{amsart} %or amsproc, or amsbook

\usepackage{amsmath,amssymb,amsthm,amsfonts,enumitem}
\usepackage{color}
\usepackage{graphicx,subcaption }
\usepackage{hyperref}
\usepackage{epstopdf}
\usepackage{bm}
\usepackage{comment}

\theoremstyle{plain} %default
\newtheorem{thm}{Theorem}[section]
\newtheorem{lem}[thm]{Lemma}

\theoremstyle{definition}
\newtheorem{defn}{Definition}[section]

\newtheorem{conj}{Conjecture}[section]

\theoremstyle{remark}

\begin{document}

\title{A note on Kalai's $3^d$ Conjecture}

\author{Gregory R. Chambers}
\address{Department of Mathematics, Rice University, Houston, TX, 77005}
\email{gchambers@rice.edu}

\author{Elia Portnoy}
\address{Department of Mathematics, MIT, Cambridge, MA, 02139}
\email{eportnoy@mit.edu}

\date{\today}
\begin{abstract}
    Suppose that $C$ is a centrally symmetric $d$-dimensional convex polytope; in 1989 Kalai conjectured that
    $C$ has at least $3^d$ faces.  We prove this result if there are $d$ hyperplanes with orthogonal normal vectors so that $C$ is symmetric about all of them.
\end{abstract}
\maketitle

Suppose that $C \subset \mathbb{R}^d$ is a centro-symmetric convex polytope with nonempty interior, that is,
\begin{enumerate}
    \item $C = -C$.
    \item $C$ has nonempty interior.
    \item $C$ is the intersection of finitely many half-spaces - sets of the form
        $$  \{ x \in \mathbb{R}^d : x \cdot n \leq c \}   $$
    for some fixed unit vector $n$ and real number $c$.
\end{enumerate}

Before stating the conjecture, we will need the following definition:
\begin{defn}
    \label{defn:face}
    Suppose that $C$ is as above.  A $d$-dimensional \emph{face} of $C$ is
    $C$ itself, and a $k$-dimensional \emph{face} of $C$ with $0 \leq k < d$ is a subset $X$ of $\partial C$ so that:
    \begin{enumerate}
        \item $X$ has Hausdorff dimension $k$.
        \item $X$ is equal to $\partial C \cap H$, where $H$ is a hyperplane.
    \end{enumerate}
\end{defn}

In 1989, Kalai made the following conjecture in \cite{kalai}:
\begin{conj}[Kalai]
    \label{conj:kalai}
    If $C$ is as above, then $C$ has at least $3^d$ faces (adding up all the faces of all dimensions).
\end{conj}

It is known that Conjecture~\ref*{conj:kalai} is true for dimensions $\leq 4$ (see \cite{4d}).  It is also known
that it is true for all polytopes whose faces are simplices; this was proved by Stanley in \cite{simplices},
answering a conjecture to due to B\'{a}r\'{a}ny and Lov\'{a}sz in \cite{stanley_precursor}.  Finally, it is known to
be true for the Hansen polytopes of split graphs, see \cite{split}.

We can now state our main theorem, for which we will give a short proof:
\begin{thm}
    \label{thm:main}
    Suppose that $C$ is as above. Let $\mathcal{B}$ be an orthogonal basis of $\mathbb{R}^d$ and suppose
    that $C$ is symmetric about each hyperplane through the origin normal to some vector in $\mathcal{B}$. Then Conjecture~\ref*{conj:kalai} is true.
\end{thm}

We would also like to point out that, at the same time that this article was completed, this result was independently proved
by R. Sanyal and M. Winter in \cite{SW}.  We now introduce the notion of a cone used our proof. 

\begin{comment}
\begin{defn}
    \label{defn:cones}
    Suppose that $K \subset \{ e_1, \dots, e_d \} \times \{ -1, 1 \}$ is nonempty, and for each $e_i$, at most one
    of $(e_i,-1)$ or $(e_i,1)$ is in $K$; we define the \emph{cone} $X_K$ by
        $$X_K = \{\sum_{(e_i,\sigma_i) \in K} t_i e_i \sigma_i : t_1, \dots, t_{|K|} \in [0,\infty)\}$$
    Here, $|K|$ is the number of elements in $K$.
\end{defn}

\begin{defn}
    \label{defn:cone_interior}
    %If $$ X_K \{ \sum_{(e_i,\sigma_i) \in K} t_i e_i \sigma_i : t_1, \dots, t_{|K|} \in [0,\infty)$$ is a cone as per Definition~\ref*{defn:cones}.
    The interior of a cone $X_K$ is defined as all points which are the linear combination of vectors from $K$
    with \emph{every} coefficient being nonzero.  The boundary of it is the collection of all points in it
    equal to a linear combination of coordinate vectors with at least one zero coefficient.
\end{defn}
\end{comment}

\begin{defn}
    \label{defn:cones} Define $\mathcal{B}' = \cup_{v \in \mathcal{B}} \{v, -v\}$. Suppose that $K$ is any nonempty subset of $\mathcal{B}'$ so that for each $v \in \mathcal{B}$, at most one of $v$ or $-v$ is in $K$. We define the \emph{cone} $X_K$ to be the set of non-negative linear combinations of vectors in $K$, that is,
    
    $$X_K = \Bigl\{ \sum_{v \in K} t_v\, v:\, t_v \ge 0 \text{ for all } v \in K \Bigr\}.$$

    \noindent The interior of the cone is

    $$X_K^{\circ} = \Bigl\{ \sum_{v \in K} t_v\, v:\, t_v > 0 \text{ for all } v \in K \Bigr\}.$$

    %and the boundary of the cone is
    %$$\partial X_K = \{\sum_{v \in K} t_v\, v:\, t_v \ge 0 \text{ for all } v \in K \text{ and } t_{v'} = 0 \text{ for some } v' \in K\}$$
\end{defn}

We also define the interior of a face of dimension at least $1$ in the obvious way, and we define the interior
of a $0$-dimensional face (a point) as the face (point) itself.  We will use $\tau^\circ$ to
denote the interior of the face $\tau$.

\begin{lem}
    \label{lem:number_cones}
    The number of cones is $3^d - 1$.
\end{lem}
\begin{proof}
    For every subset of $\mathcal{B}$ of size $k$, there are $2^k$ cones.  Thus, the total number of cones is
        $$ \sum_{k=1}^d { d \choose k } 2^k $$
    which by the binomial theorem is $3^d - 1$.
\end{proof}

For every cone $X_K$, we will associate a face $\tau_K \subset \partial C$ to it.  We will prove that they are all distinct; since there is exactly one $d$-dimensional simplex ($C$ itself), Lemma~\ref*{lem:number_cones} will complete the proof of Theorem~\ref*{thm:main}.

Fix a cone $X_K$, and choose $\tau_K$ to be a face of $\partial C$ which satisfies
    $$ \tau^{\circ} \cap X_K^{\circ} \neq \emptyset, $$
and which has minimal dimension among all of these. Note that such a $\tau_K$ exists because the interior of the $d$-dimensional face of $C$ intersects the interior of each cone, and so $\partial C$ intersects the interior of each cone.

Denote by $Q_K$ the union of all cones that contain $X_K$. Since $\mathcal{B}$ is an orthonormal basis, $Q_K$ can also be defined as

$$Q_K = \bigcap_{v \in K} \{x \in \mathbb{R}^d: x \cdot v \ge 0\}$$

%We also define $K^*$ to be all unit coordinate vectors in $K$, that is,
%    $$ K^* = K \cap \{\pm e_1, \dots, \pm e_d \}. $$

We have the following lemma:
\begin{lem}
    \label{lem:interior_inclusion}
    We have the following inclusion:
        $$ \tau_K^{\circ} \subset Q_K^\circ. $$
\end{lem}
\begin{proof}
    
    For contradiction, assume that the interior of $\tau_K$ is \emph{not} a subset of the interior of $Q_K$.  
    Note that every cone is contained in some $d$-dimensional cone, so $Q_K$ is the closure of an open set in $\mathbb{R}^d$. Thus, as $\tau_K^{\circ}$ is path-connected and contains a point in $X_K \subset Q_K$, there is some point $q \in \partial Q_K \cap \tau_K^{\circ}$. Note that 

    $$\partial Q_K \subset \bigcup_{v \in K} \{x \in \mathbb{R}^d: x \cdot v = 0\}$$

    Thus, there is some $v_q \in K$ such the $q$ lies in the hyperplane $H$ normal to $v_q$. Let
    $R: \mathbb{R}^d \to \mathbb{R}^d$ denote the reflection about $H$; $R(q) = q$ and $C$ is symmetric about
    $H$ by assumption. Since $q$ is an interior point of $\tau_K$, it follows that $R(\tau_K) = \tau_K$. 
    
    Now consider $p \in \tau_K^{\circ} \cap X_K^{\circ}$, which exists by the definition of $\tau_K$.
    By convexity and the fact that $\tau_K$ is symmetric about $H$, $\tau_K$ contains the segment from $R(p)$ to $p$ in the direction $v_q$. $R$ does not fix any point in $X_K^{\circ}$ and
    $p \in X_K^{\circ}$, so this segment has non-zero length.
    Now let $L$ be the ray based at $p$ and in the direction $v_q$. Note that $L \subset X_K^{\circ}$ since $v_q \in K$ and $p\in X_K^{\circ}$.
    Since $\tau_K$ contains the segment just described, $\partial \tau_K \cap L$ is not empty. Thus,
    $\partial \tau_K$ contains a face $\rho_K$ with $\rho_K^{\circ} \cap X_K^{\circ} \neq \emptyset$. Since $\tau_K$ contains a segment of non-zero length, it has dimension at least $1$ and so $\dim(\rho_K) < \dim(\tau_K)$.
    This contradicts the minimality of the dimension of $\tau_K$, and so we have shown that $\tau_K^{\circ} \subset Q_K^{\circ}$.
\end{proof}

We can now prove the main theorem.

\begin{proof}[Proof of Theorem~\ref*{thm:main}]
It suffices to show that $\tau_K = \tau_{K'}$ if and only if $K = K'$. Suppose $K \neq K'$ and, without loss of generality, that $\dim(X_K) \ge \dim(X_{K'})$. Then there is some $v \in K$ with $v \not\in K'$. So for any $p' \in X_{K'}$ we have $p' \cdot v \le 0$. This means that $X_{K'}$ is disjoint from $Q_K^{\circ}$. Therefore, by the previous lemma $\tau_K^{\circ}$ is disjoint from $X_{K'}$ and so $\tau_K \neq \tau_{K'}$.  This completes the proof of Theorem~\ref*{thm:main}.
\end{proof}

\noindent {\bf Acknowledgments} The first author was supported in part by NSF
grant DMS-1906543.  The authors would like to thank the referee for a previous version for pointing out an error in the proof.

\bibliographystyle{amsplain}
\bibliography{bibliography}

\end{document}